\documentclass[A4paper,reqno,oneside,12pt]{amsart}
\usepackage[foot]{amsaddr}
\usepackage{fullpage}
\usepackage{setspace}
\usepackage{amsmath, amssymb, amsfonts, amsthm, bbm, mathtools, mathrsfs}
\usepackage{tikz}
\usepackage{subcaption}
\usepackage{verbatim,enumitem,graphics}

\usepackage{xcolor}
\usepackage[backref=page]{hyperref}
\usepackage{cleveref}
\hypersetup{%
    colorlinks,
    linkcolor={red!50!black},
    citecolor={green!50!black},
    urlcolor={blue!50!black}
}
\usepackage{dsfont}

\usepackage{lineno}
\usepackage[normalem]{ulem}

\usepackage{orcidlink}

\allowdisplaybreaks

\theoremstyle{plain}

\newtheorem{theorem}{Theorem}[section]
\newtheorem*{theorem*}{Theorem}
\newtheorem{lemma}[theorem]{Lemma}

\newtheorem{proposition}[theorem]{Proposition}

\newtheorem{corollary}[theorem]{Corollary}
\newtheorem{conjecture}[theorem]{Conjecture}

\theoremstyle{definition}
\newtheorem{remark}[theorem]{Remark}
\newtheorem*{remark*}{Remark}
\newtheorem{definition}[theorem]{Definition}

\newcommand{\abs}[1]{\left| #1 \right|}
\newcommand{\of}[1]{\left( #1 \right)}

\DeclareMathOperator{\pl}{\mathrm{pl}}
\DeclareMathOperator{\mG}{\mathcal{G}}
\newcommand{\set}[1]{\left\{ #1 \right\}}

\begin{document}

\setstretch{1.27}

\title{Chromatic polynomial evaluation spectra}

\author[Rafael Miyazaki]{Rafael Miyazaki\:\orcidlink{0009-0008-4548-6290}}
\author[Cosmin Pohoata]{Cosmin Pohoata\:\orcidlink{0009-0004-4842-2939}}
\author[Michael Zheng]{Michael Zheng\:\orcidlink{0009-0001-2406-6130}}
\address{Department of Mathematics, Emory University, Atlanta, GA, 30322, USA}
\email{\{rafael.kazuhiro.miyazaki, cosmin.pohoata, xzhe226\}@emory.edu}

\maketitle

\medskip
\begin{abstract}
Around 10 years ago, Agol and Krushkal showed that the number of chromatic polynomials $P_{G}$ arising from graphs $G$ on $n$ vertices grows exponentially with $n$, by establishing that the (dual) flow polynomial $F_{G}\left(\frac{3+\sqrt{5}}{2}\right)$ already takes on exponentially many values, if one varies $G$ over all planar cubic graphs $G$ on $n$ vertices. 

We show, more generally, that the size of the set $\{P_G(q): |V(G)|=n\}$ is exponential in $n$, for every fixed real number $q \neq 0,1,2$. In fact, our approach can also be pushed to show that $P_{G}(q)$ already takes on exponentially many values, if we only vary $G$ over all planar graphs on $n$ vertices. The case $q=3$ confirms a conjecture of Agol, which was initially motivated by the {\sf NP}-completeness of planar $3$-colorability. 
\end{abstract}

\section{Introduction} \label{sec:intro}

A \emph{proper $q$--coloring} of a finite undirected graph $G=(V,E)$ is an assignment
$\varphi:V\to[q]$ such that adjacent vertices receive different colors.
When $q\in\mathbb{Z}_{\ge 1}$, the number of proper $q$--colorings is denoted $P_G(q)$, and it is a classical theorem of Birkhoff and Whitney that there is a unique polynomial $P_G(x)\in\mathbb{Z}[x]$ which encodes the number of proper $q$-colorings of $G$ when evaluated at $x=q$. This is the so-called {\it{chromatic polynomial}} of graph $G$, a polynomial of degree $|V|$ and with leading coefficient equal to $1$. This object was first introduced by Birkhoff~\cite{Birkhoff1912} in 1912 as part of an algebraic approach for the Four Color Problem. Birkhoff’s original motivation was to study colorings of planar maps by encoding them in a polynomial and analyzing its behavior at small integer values, particularly at $x=4$. Although this approach did not yield a proof of the Four Color Theorem, it initiated a rich line of research and established the chromatic polynomial as a fundamental graph invariant. Whitney~\cite{Whitney1932} subsequently extended Birkhoff’s ideas to general graphs and introduced many of the structural properties of chromatic polynomials that are now standard. Among these is the deletion--contraction identity: for any edge $e$ of $G$,
\begin{equation}
    P_G(x)=P_{G-e}(x)-P_{G/e}(x), \label{eq:dc-intro}
\end{equation}
where $G-e$ denotes deletion and $G/e$ denotes contraction.
This identity places chromatic polynomials within the broader framework of the Tutte polynomial~\cite{Tutte1954}.

The chromatic polynomial encodes substantial information about $G$.
For instance, the chromatic number satisfies
\[
\chi(G)=\min\{q\in\mathbb{Z}_{\ge 1}: P_G(q)>0\}.
\]
Evaluations at negative integers also have a rich combinatorial meaning:
Stanley~\cite{Stanley1973} showed that
\[
|P_G(-1)|=\#\{\text{acyclic orientations of }G\},
\]
and more generally his chromatic reciprocity theorem interprets $(-1)^{|V|}P_G(-m)$, for $m\in\mathbb{Z}_{\ge 1}$,
as a weighted count of acyclic orientations together with order-preserving maps to $[m]$. 

\medskip

\noindent {\bf{Chromatic polynomial spectrum.}} In \cite{AKTutte}, Agol and Krushkal proved that the number of chromatic polynomials of planar triangulations on $n$ vertices grows exponentially with $n$. In a few words, their argument works with duals of planar cubic graphs and considers the number of distinct values taken by the \emph{flow polynomial} at the special value
$Q=(3-\sqrt5)/2$ (for a planar graph $G$, its flow polynomial is precisely $P_{G^{*}}(x)$, where $G^{*}$ is the dual of $G$). At this particular parameter $Q$, the chromatic algebra (a certain algebra over $\mathbb{C}[Q]$ formed from linear combinations of isotopy classes of planar graphs in a rectangle with a fixed number of univalent points on top and bottom) satisfies an additional local relation (coming from Temperley--Lieb / Jones--Wenzl theory),
which collapses a natural boundary module from dimension three to dimension two.
Two explicit local operations on planar graphs $G$ then act on $P_{G^{*}}(Q)$ by $2\times2$ matrices whose squares generate a free semigroup,
yielding exponentially many distinct evaluations and hence exponentially many distinct chromatic polynomials.
See~\cite{AKFlow,AKTutte} for more details and further references. 

In this paper we study the asymptotics of the \emph{evaluation spectra} of the chromatic polynomial. First, let $\mG_n$ be the set of all simple graphs on $n$ vertices. For each fixed real number $q$ and integer $n\ge 1$, we can now define
\[
\mathsf{S}_n(q)\;\coloneqq \;\{\,P_G(q)\,:\,G \in \mG_n\,\}.
\]
In words, $\mathsf{S}_n(q)$ represents the set of distinct values attained by $P_G(q)$ as $G$ ranges over all simple graphs on $n$ vertices. Furthermore, let 
\begin{equation*}
    \mathsf{S}_n^{\pl}(q) \; \coloneqq \; \{\, P_G(q) \, : \, G \in \mG_n^{\pl} \, \}
\end{equation*}
where $\mG_n^{\pl}$ is the set of all $n$--vertex, simple, planar graphs. Trivially, $\abs{\mathsf{S}_n(q)} \geq \abs{\smash{\mathsf{S}_n^{\pl}}(q)}$. 

\medskip
\noindent\textbf{Our results.}
We establish lower bounds for $\mathsf{S}_n(q)$ and $\mathsf{S}_n^{\pl}(q)$ for arbitrary fixed $q$. For integer $q\ge 3$ the value $P_G(q)$ counts proper $q$--colorings, and Agol conjectured that the spectrum $\mathsf{S}_n^{\pl}(3)$ should also grows exponentially over $\mG_n^{\pl}$; see~\cite{TreumannMO} for more context, in particular a motivation coming from the {\sf NP}-completeness of $3$--colorability of planar graphs. 
Our main theorem confirms this conjecture and, more generally, shows that the evaluation spectrum is exponential in $n$ for every fixed
$q\notin\{0,1,2\}$.

\begin{theorem}\label{thm:planar}
For every real number $q\notin\{0,1,2\}$,
the evaluation spectrum $|\mathsf{S}_n^{\pl}(q)|$ grows exponentially in $n$. 

Moreover, we can show the following uniform lower bounds: 
\[
|\mathsf{S}^{\pl}_n(q)| \;\ge\;
\begin{cases}
2^{(n-2)/2}, & q<0,\\[4pt]
2^{n/8 - o(n)}, & \hspace{-4mm} 0 < q < 2,\\[4pt]
\sqrt{F_{n-2}}= \Theta(\varphi^{(n-2)/2}), & q>2,\
\end{cases}
\]
where $(F_n)_n$ is the Fibonacci sequence and $\varphi = (1+\sqrt{5})/2$ is the golden ratio.
\end{theorem}

In particular, Theorem \ref{thm:planar} and the trivial inequality $|\mathcal{P}_n| \geq \abs{\mathsf{S}_n(-1)} \geq \abs{\smash{\mathsf{S}_n^{\pl}}(-1)}$ together imply that $|\mathcal{P}_n| \geq 2^{(n-2)/2}$. This fact also improves quantitatively upon the lower bound from \cite{AKFlow}. Furthermore, the lower bounds from Theorem \ref{thm:planar} hold true even if we additionally require the family of (planar) graphs in $\mG_n^{\pl}$ to be connected. See Remark \ref{rmk:2-connec} for more details.

Before we say more about the proof of Theorem \ref{thm:planar}, let us first remark that for $q\in\{0,1,2\}$ the spectrum is not exponential even if we allow $G$ to vary over {\it{all}} graphs on $n$ vertices.
Indeed,
\begin{itemize}
\item $P_G(0)=0$ for every graph with at least one vertex, so $|\mathsf{S}^{\pl}_n(0)|=|\mathsf{S}_n(0)|=1$;
\item $P_G(1)=1$ if $G$ is edgeless and $P_G(1)=0$ otherwise, so $|\mathsf{S}^{\pl}_n(1)|=1$ and $|\mathsf{S}_n(1)|= 2$;
\item $P_G(2)$ counts proper $2$--colorings, so it is $0$ for non-bipartite $G$, and equals $2^{c(G)}$ for bipartite $G$ with $c(G)$ connected components.
Hence $|\mathsf{S}^{\pl}_n(2)|\leq |\mathsf{S}_n(2)|= n+1$.
\end{itemize}

The proof of Theorem \ref{thm:planar} is more delicate when $q \in (0,2)$ (and precise quantitative bounds are also more difficult to extract), so, as a prequel to that side of the story, we will first use an elementary argument to show that $|\mathsf{S}_n(q)|$ is exponential in that range, by reducing the problem to the case when $q < 0$. We record this as a second theorem, in greater generality. 

\begin{theorem}\label{thm:main}
For every non-integer real number $q > 0$, and $n\ge \lceil q\rceil+2$, 
$$|\mathsf{S}_n(q)| \;\ge\; 2^{(n-\lceil q\rceil-2)/2}.$$
\end{theorem}

In particular, if $q \in (0,2)$, then Theorem \ref{thm:main} immediately implies that $|\mathsf{S}_n(q)| \ge 2^{(n-4)/2}$ for all $n\ge 4$. 

Last but not least, let us also emphasize that for positive integers $q \not\in \left\{0,1,2\right\}$, exponential growth for $\abs{\mathsf{S}_n(q)}$ is best possible, up to the correct base of the exponent. Indeed, if $q\in\mathbb{Z}_{\ge 1}$ then every $n$--vertex graph has at most $q^n$ proper $q$--colorings, so $|\mathsf{S}_n(q)|\le q^n+1$. 

When $q$ is \emph{not} a nonnegative integer, there is however no comparable a priori bound of the form $|P_G(q)|\le C^n$.
When $q=-1$, Stanley's theorem~\cite{Stanley1973} gives $|P_{K_n}(-1)|=n!$, and, more generally, it is not difficult to show that for any fixed $q\in\mathbb{R}\setminus\mathbb{Z}_{\ge 0}$ the complete graph satisfies
$P_{K_n}(q)=(q)_n=q(q-1)\cdots(q-n+1)$, whose magnitude grows on the order of $n!$.
These factorial-scale values suggest that, at least for some non-integer or negative points, the spectrum itself should be \emph{superexponential}. We will discuss this some more in Section \ref{sec:conclusion}.

\vspace{+3mm}
\noindent\textbf{Related work and proof ideas.}
A closely related spectrum problem was recently resolved for the spanning-tree enumerator.
Let $\tau(G)$ be the number of spanning trees of $G$, and set
\[
\mathcal{T}(n)\coloneqq \{\tau(G): G \in \mG_n \},
\]
as well as
\begin{equation*}
    \mathcal{T}^{\pl}(n) \coloneqq \{\, P_G(q) \, : \, G \in \mG_n^{\pl} \}.
\end{equation*}

In \cite{ChanKontorovichPak24trees}, Chan, Kontorovich and Pak proved that $|\mathcal{T}^{\pl}(n)|$ grows exponentially in $n$, via a connection to continued fractions, thin orbits, and a deep result of Bourgain-Kontorovich \cite{BourgainKontorovich14} on the so-called Zaremba conjecture \cite{Zaremba72}. See \cite{ChanKontorovichPak24trees} and the references therein for more information. Soon after, Alon--Buci\'c--Gishboliner~\cite{ABG} gave a much simpler alternative proof, which also came with improved quantitative bounds and results for other classes of graphs. Both arguments made crucial use of the deletion--contraction formula for the spanning-tree enumerator $\tau$ and a handful of basic operations that increase the number of vertices of the graph, and which allow one to track the evolution of the parameters $\tau(G-e)$ and $\tau(G/e)$ (for various edges $e$ of $G$), as the graph $G$ undergoes these changes. In \cite{ABG}, however, Alon, Buci\'c and Gishboliner noted that the parameters $\tau(G-e)$ and $\tau(G/e)$ of $G$ can sometimes instantly uniquely decode the entire process, an observation which drastically simplified the analysis. 

For the chromatic polynomial spectrum, Agol and Krushkal \cite{AKTutte} also previously employed a similar strategy to show that $|\mathcal{P}_n|$ grows exponentially with $n$. An important difference in their argument is that the basic operations employed are less simple than the ones in ~\cite{ABG} or \cite{ChanKontorovichPak24trees}, and so one cannot readily decode the process if their operations are used freely to construct their $n$--vertex graphs. In order to circumvent this barrier, they consider only graphs obtained through a sequence of block operations, each of which combines some of their set of basic operations. To decode the sequence of block operations from the final parameters obtained, Agol and Krushkal then appeal to a simple application of the ping--pong lemma, a basic tool from geometric group theory. See for example \cite{Loh} and the references therein or \cite{BrennerCharnow78}.

Our argument, in some sense, unifies the two strategies from \cite{AKTutte} and \cite{ABG}. To prove Theorem \ref{thm:planar}, we make use of block operations that consist of sequences of two of the basic operations used in~\cite{ABG}, and then establish various ping--pong lemma type criteria, in order to decode.

\vspace{+2mm}

\noindent\textbf{Organization of the paper.} The proof of Theorem \ref{thm:planar} is most pleasant when $q<0$, so we set the stage by discussing this case first in Section \ref{sec:negative}.

After that, in Section \ref{sec:noninteger}, we take a first dip into the more problematic case when $q \in (0,2)$, which we will first address for the $|\mathsf{S}_n(q)|$. Using a different kind of combinatorial argument, we will derive a lower bound by reducing the problem to the case when $q \in (-1,0)$, for which the result from Section \ref{thm:planar} applies. This argument is the content of Theorem \ref{thm:main} and works in greater generality for every non-integer real number $q > 0$, and $n\ge \lceil q\rceil+2$. It does not apply directly to lower bound $|\mathsf{S}^{\pl}_n(q)|$ but it provides a better quantitative bound for $|\mathsf{S}_n(q)|$ than Theorem \ref{thm:planar} does in this range. 

In Section \ref{sec:integers}, we continue the proof of Theorem \ref{thm:planar} with a more complicated ping--pong argument showing that $|\mathsf{S}^{\pl}_n(q)|$ is exponential in $n$ for each $q >2$. 

In Section \ref{sec:planar_0_2}, we then complete the proof of Theorem \ref{thm:planar} by showing that $|\mathsf{S}^{\mathrm{pl}}_n(q)|$ is exponential in $n$ for every $q\in(0,2)\setminus\{1\}$ as well, by using a more intricate combination of block constructions. 

In Section \ref{sec:conclusion}, we end with some concluding remarks and a couple of natural open problems. 

\section*{Acknowledgments}
The last two authors are also grateful to Matija Buci\'c for an excellent talk on \cite{ABG} at the IAS/PCMI summer program on extremal combinatorics, which inspired the present paper. C.P.\ was supported by NSF grant DMS-2246659.

\section{Negative evaluations and chromatic reciprocity}\label{sec:negative}

Throughout this section we fix a real number $q<0$ and write $\lambda\coloneqq-q>0$. 

\subsection{A sign-normalization and what it counts at negative integers}

Define the sign-normalized evaluation
\begin{equation}\label{eq:Zdef}
Z_G(\lambda)\coloneqq(-1)^{|V(G)|}\,P_G(-\lambda)=(-1)^{|V(G)|}\,P_G(q).
\end{equation}
For a fixed $n$, multiplying by the constant factor $(-1)^n$ is injective, so bounding the number of distinct values of $Z_G(\lambda)$ over $|V(G)|=n$
is equivalent to bounding $|\mathsf{S}_n(q)|$.

\begin{proposition}[Chromatic reciprocity {\cite[Thm.\ 1.2]{Stanley1973}}]\label{prop:chromrecip}
Let $G$ be a loopless graph on $n$ vertices and let $m\in\mathbb{Z}_{\ge 1}$.
Then $Z_G(m)=(-1)^nP_G(-m)$ equals the number of pairs $(\mathcal{O},f)$ where
$\mathcal{O}$ is an acyclic orientation of $G$ and $f:V(G)\to[m]$ is order-preserving with respect to $\mathcal{O}$ (i.e., if $u\to v$ is an oriented edge of $\mathcal{O}$, then $f(u)\le f(v)$).
In particular, $Z_G(1)=\#\{\text{acyclic orientations of }G\}$.
\end{proposition}

\begin{remark}\label{rem:noninteger}
For non-integer $\lambda>0$ there is, in general, no literal set being counted by $Z_G(\lambda)$; rather $Z_G(\lambda)$ is the polynomial extension of the counting function in Proposition~\ref{prop:chromrecip}.
Our proof below uses only the structural consequences that persist for all $\lambda>0$: positivity and an additive deletion--contraction recurrence.
(For a geometric perspective on these reciprocity phenomena, see~\cite{BeckZaslavskyInsideOut}.)
\end{remark}

\subsection{Positivity and an additive deletion--contraction identity}

A standard consequence of Whitney's broken-circuit theorem (see, e.g.,~\cite{StanleyEC1}) is that the coefficients of $P_G(x)$ alternate in sign:
for a loopless graph $G$ on $n$ vertices there exist integers $a_0,\dots,a_n\ge 0$ such that
\begin{equation}\label{eq:alt-sign}
P_G(x)=\sum_{k=0}^n (-1)^k a_k x^{n-k}.
\end{equation}

\begin{lemma}[Positivity for $\lambda>0$]\label{lem:Zpositive}
If $G$ is loopless and $\lambda>0$, then $Z_G(\lambda)>0$.
\end{lemma}

\begin{proof}
Substituting $x=-\lambda$ into~\eqref{eq:alt-sign} gives
\[
(-1)^n P_G(-\lambda)=\sum_{k=0}^n a_k \lambda^{n-k},
\]
which is strictly positive since $a_0=1$.
\end{proof}

\begin{lemma}[Additive deletion--contraction for $Z$]\label{lem:dcplus}
Let $e$ be a non-loop edge of $G$. Then
\begin{equation}\label{eq:dcplus}
Z_G(\lambda)=Z_{G-e}(\lambda)+Z_{G/e}(\lambda).
\end{equation}
\end{lemma}

\begin{proof}
From deletion--contraction $P_G(x)=P_{G-e}(x)-P_{G/e}(x)$ and the definition~\eqref{eq:Zdef} we compute
\[
Z_G(\lambda)
=
(-1)^{|V(G)|}P_G(-\lambda)
=
(-1)^{|V(G)|}P_{G-e}(-\lambda)-(-1)^{|V(G)|}P_{G/e}(-\lambda).
\]
The first term is $Z_{G-e}(\lambda)$, which has the same number of vertices, while
$$(-1)^{|V(G)|}P_{G/e}(-\lambda)=-(-1)^{|V(G/e)|}P_{G/e}(-\lambda)=-Z_{G/e}(\lambda),$$
giving~\eqref{eq:dcplus}.
\end{proof}

\subsection{Feasible vectors and a square-root lemma}

\begin{definition}[Feasible vectors]\label{def:neg-feasible}
Let $e$ be an edge of a planar graph $G$.  Define the feasible vector witnessed by $(G,e)$ to be
\[
v_\lambda(G,e)\coloneqq
\begin{pmatrix}
Z_{G/e}(\lambda)\\
Z_{G-e}(\lambda)
\end{pmatrix}.
\]
We write $v(G,e)= v_\lambda(G,e)$ when $\lambda$ is clear from context. We call $v$ an \emph{$n$--feasible} vector if there exists an $n$--vertex planar graph $G$ and an edge $e\in E(G)$ such that $v(G,e)=v$. 
\end{definition}

We record a simple yet crucial lemma connecting the number of such vectors with size of $\mathsf{S}^{\pl}_n (q)$. 

\begin{lemma}[Square-root lemma]\label{lem:sqrt-neg}
If there are at least $N$ distinct $n$--feasible vectors, then
\[
|\mathsf{S}^{\pl}_n (q)| \;\ge\; \sqrt{N}.
\]
\end{lemma}

Lemma \ref{lem:sqrt-neg} is completely analogous to \cite[Lemma 2.1]{ABG}. 

\begin{proof}
Let $(t_i,u_i)$ be $N$ distinct $n$--feasible vectors with witnesses $(G_i,e_i)$.
Consider the multiset of second coordinates $\{u_i\}$.

If at least $\sqrt{N}$ of the $u_i$ are distinct, then, since edge-deletion preserves planarity, the values $Z_{G_i-e_i}(\lambda)=u_i$ give at least $\sqrt{N}$ distinct evaluations coming from $n$--vertex planar graphs.

Otherwise, some value $u$ occurs for at least $\sqrt{N}$ indices.  For those indices the pairs $(t_i,u)$ are distinct, hence the sums $t_i+u$ are distinct.
By Lemma~\ref{lem:dcplus}, $t_i+u = Z_{G_i}(\lambda)$, so again we obtain at least $\sqrt{N}$ distinct values.
Finally $P_G(q)=(-1)^n Z_G(\lambda)$ for $|V(G)|=n$, so the same lower bound holds for $|\mathsf{S}_n^{\pl} (q)|$.
\end{proof}

\subsection{Two basic operations and their matrices}

We use two operations on witnessed edges.  Let $e=ab$ be an edge of $G$.

\begin{definition}[Basic operations]\label{def:ops-neg}
\begin{enumerate}
\item[(S)] (\emph{Subdivision}) Subdivide $e$ by a new vertex $w$, replacing $e$ with the path $a\!-\!w\!-\!b$.
Let $G'$ be the new graph and set $S(G,e)\coloneqq(G',aw)$.
\item[(B)] (\emph{Add a vertex adjacent to both endpoints}) Add a new vertex $w$ adjacent to both $a$ and $b$.
Let $G^\ast$ be the new graph and set $B(G,e)\coloneqq(G^\ast,e)$.
\end{enumerate}
\end{definition}

We repeatedly use the standard leaf identity: if $H$ is obtained from $G$ by attaching a new leaf to an existing vertex, then
\begin{equation}\label{eq:leaf}
P_H(x)=(x-1)P_G(x)\qquad\text{and hence}\qquad Z_H(\lambda)=(\lambda+1)Z_G(\lambda).
\end{equation}

\begin{lemma}[Matrix action]\label{lem:matrices-neg}
Let $e$ be an edge of a planar graph $G$, and write $v(G,e)=\binom{t}{u}$.
Then
\[
(v\circ S)(G,e)
=
\begin{pmatrix}
1 & 1\\
0 & \lambda+1
\end{pmatrix}
\binom{t}{u},
\qquad
(v\circ B)(G,e)
=
\begin{pmatrix}
\lambda+1 & 0\\
1 & \lambda+2
\end{pmatrix}
\binom{t}{u}.
\]
\end{lemma}

\begin{proof}
Let $e=ab$.

\smallskip\noindent\emph{Operation $S$.}
Let $(G',f)=S(G,e)$ with $f=aw$.
Contracting $f$ merges $a$ and $w$ and collapses the path $a\!-\!w\!-\!b$ back to the single edge $ab$, so $(G'/f)\cong G$ and hence
\[
Z_{G'/f}(\lambda)=Z_G(\lambda)=Z_{G/e}(\lambda)+Z_{G-e}(\lambda)=t+u
\]
by Lemma~\ref{lem:dcplus}.
Deleting $f$ makes $w$ a leaf attached to $b$ in the graph $G-e$, so by~\eqref{eq:leaf}
\[
Z_{G'-f}(\lambda)=(\lambda+1)Z_{G-e}(\lambda)=(\lambda+1)u.
\]
This gives the stated matrix for $S$.

\smallskip\noindent\emph{Operation $B$.}
Let $(G^\ast,e)=B(G,e)$ with new vertex $w$ adjacent to both $a$ and $b$.
Contracting $e$ in $G^\ast$ merges $a$ and $b$; the two edges $aw$ and $bw$ become parallel edges from the merged vertex to $w$, which do not affect proper colorings, so $G^\ast/e$ is obtained from $G/e$ by attaching a leaf. Hence by~\eqref{eq:leaf},
\[
Z_{G^\ast/e}(\lambda)=(\lambda+1)Z_{G/e}(\lambda)=(\lambda+1)t.
\]

For deletion, set $H\coloneqq G^\ast-e$.  We claim the polynomial identity
\begin{equation}\label{eq:polyid}
P_H(x)=P_{G/e}(x)+(x-2)P_{G-e}(x).
\end{equation}
Indeed, apply deletion--contraction to the edge $aw$ in $H$.
Deleting $aw$ leaves a leaf at $b$ attached to $G-e$, hence $P_{H-aw}(x)=(x-1)P_{G-e}(x)$.
Contracting $aw$ identifies $a$ with $w$ and turns $bw$ into an edge $ab$, recovering $G$ (up to harmless parallel edges),
so $P_{H/aw}(x)=P_G(x)=P_{G-e}(x)-P_{G/e}(x)$.
Thus, $P_H(x)$ equals
\[
P_{H-aw}(x)-P_{H/aw}(x)
=
(x-1)P_{G-e}(x)-\bigl(P_{G-e}(x)-P_{G/e}(x)\bigr)
=
P_{G/e}(x)+(x-2)P_{G-e}(x),
\]
proving~\eqref{eq:polyid}.
Evaluating~\eqref{eq:polyid} at $x=-\lambda$ and multiplying by $(-1)^{|V(H)|}=(-1)^{|V(G)|+1}$ yields
\[
Z_H(\lambda)=Z_{G/e}(\lambda)+(\lambda+2)Z_{G-e}(\lambda)=t+(\lambda+2)u.
\]
This gives the stated matrix for $B$.
\end{proof}

In view of this lemma, we will sometimes write, say, $S$ for the matrix corresponding to the operation $S$. 

\subsection{Ping--pong on ratios and exponential growth}

For $v=\binom{x}{y}$ with $y \neq 0$, define $r(v)\coloneqq x/y$. We also define the feasible vector witnessed after a sequence of operations as follows.
\begin{definition}\label{def:feasible_generated}
 Let $a = (a_1,\dots,a_t)$ be a word in the semigroup $\{B,S\}^\ast$, $G$ be a planar graph, and $e$ be an edge of $G$. Then define 
     \[\mathbf{v}(a,G,e) \coloneqq (v\circ a_t\circ a_{t-1}\circ\dots\circ a_1)(G,e),\]
     the feasible vector witnessed by the pair generated through the sequence of operations encoded by $a$ starting with the pair $(G,e)$.
\end{definition}
We can now prove the following lemma, which guarantees that different words of the same length generate different feasible vectors if starting from an appropriate pair $(G,e)$.
\begin{lemma}\label{lem:pingpong-neg}
Let $v_0\coloneqq v(K_2,e)$ where $e$ is the unique edge of $K_2$.
Then $v_0=\binom{\lambda}{\lambda^2}$ and $r(v_0)=1/\lambda$.
Moreover, on ratios the operations act by
\[
f_S(r)=\frac{r+1}{\lambda+1},
\qquad
f_B(r)=\frac{(\lambda+1)r}{r+\lambda+2}.
\]
Let $I\coloneqq (0,1/\lambda]$. Then $f_S(I)\subset\bigl(1/(\lambda+1),\,1/\lambda\bigr]$ and $f_B(I)\subset\bigl(0,\,1/(\lambda+1)\bigr]$,
so the images are disjoint and contained in $I$.
Consequently, if $w\neq w'$ are words in the semigroup $\{S,B\}^\ast$ with the same number of letters, then $r(\mathbf{v}(w,K_2,e))\neq r(\mathbf{v}(w',K_2,e))$.
\end{lemma}

\begin{proof}
The base computation is immediate: $K_2/e$ is a single vertex graph, so $Z_{K_2/e}(\lambda)=\lambda$; $K_2-e$ has two isolated vertices, so $Z_{K_2-e}(\lambda)=\lambda^2$.

From Lemma~\ref{lem:matrices-neg}, if $v=\binom{x}{y}$ then
\[
Sv=\binom{x+y}{(\lambda+1)y},\qquad Bv=\binom{(\lambda+1)x}{x+(\lambda+2)y},
\]
so the ratio maps are $f_S,f_B$.

Now fix $r\in I$, so $0<r\le 1/\lambda$. Since $f_S$ is strictly increasing,
\[
f_S(r)>f(0) =\frac{1}{\lambda+1}
\quad\text{and}\quad
f_S(r)\le f_S(1/\lambda) =\frac{1}{\lambda}.
\]
Thus $f_S(I)\subset(1/(\lambda+1),\,1/\lambda]$.
Similarly, since $f_B$ is strictly increasing,
\[
 f_B(r) > f_B(0) = 0\quad\text{and}\quad f_B(r)\le f_B(1/\lambda)=\frac{1}{\lambda+1}.
\]
Thus $f_B(I)\subset(0,\,1/(\lambda+1)]$.
Hence, the images are disjoint.

Finally, suppose $r(\mathbf{v}(w,K_2,e))= r(\mathbf{v}(w',K_2,e))$.
Observe that they must both lie in $I$.
By the image 
disjointness just proved and the fact that they have the same number of letters, they are either both empty or the last letter of $w$ must equal the last letter of $w'$. In the former case, $w=w'$, while in the latter case, since both matrices in Lemma~\ref{lem:matrices-neg} are invertible, we can cancel the last letter in each word and iterate to conclude that $w=w'$.
\end{proof}

\begin{remark}
    \label{rmk:2-connec}
    Note that the operations defined in Definition \ref{def:ops-neg} preserve graph planarity and connectivity, and increase the number of vertices of the graph by one. In fact, if we apply a sequence of these operations starting with $(K_3, e)$ or $(K_4,e)$, $e$ being an arbitrary edge of the respective graph, then it is not hard to inductively show that the resulting graph is $2$--connected. Indeed, this follows from Whitney's theorem \cite{WhitneyEar}, since the graph will always have an ear decomposition: after applying $S$, replace $e$ by $a-w-b$ within the ears containing $e$; after applying $B$, we simply gain a new ear.    

    In particular, replacing the starting graph in Lemma~\ref{lem:pingpong-neg} by $(K_3, e)$, Lemma~\ref{lem:pingpong-neg}, and the subsequent Lemma ~\ref{lem:ping_pong_02} and Lemma~\ref{lem:pingpong-pos}, all yield $2$--connected, planar witnesses, so after applying Lemma \ref{lem:sqrt-neg} and Lemma \ref{lem:sqrt-pos} respectively, we get the same bound asymptotically on the evaluation spectrum even if we require $G\in \mG_n^{\pl}$ to be connected.  
\end{remark}

We can now prove the first bound of Theorem~\ref{thm:planar}.

\begin{theorem}\label{thm:negative}
Fix a real number $q<0$. Then for every $n\ge 2$,
\[
|\mathsf{S}^{\pl}_n(q)| \;\ge\; 2^{(n-2)/2}.
\]
\end{theorem}

\begin{proof}
Starting from $(K_2,e)$ on $2$ vertices, each application of $S$ or $B$ adds exactly one vertex and preserves planarity.
Thus every word $w\in\{S,B\}^{n-2}$ yields an $n$--vertex witness and an $n$--feasible vector
$\mathbf{v}(w, K_2,e)$. 
By Lemma~\ref{lem:pingpong-neg}, these $2^{n-2}$ vectors are all distinct.
Applying Lemma~\ref{lem:sqrt-neg} finishes the proof:
\[
|\mathsf{S}^{\pl}_n(q)|\ge \sqrt{2^{n-2}}=2^{(n-2)/2}.\qedhere
\]
\end{proof}

\begin{remark}
    It is not a coincidence that Theorem \ref{thm:negative} recovers the same quantitative bound from \cite[Corollary 2.6]{ABG}. Our two basic operations $S$ and $B$ are two out of the three operations used in \cite{ABG}, and our ping--pong Lemma \ref {lem:pingpong-neg} morally substitutes their \cite[Theorem 2.5]{ABG} in the analysis. 
    
    It is perhaps important to add that the third operation used in \cite{ABG} to subsequently upgrade their lower bound does not behave well for our purposes. Namely, in the notation of Definition \ref{def:neg-feasible} and \ref{def:ops-neg}, if we consider the witness $(G^\ast, wa)$, then one can check that 
\begin{equation*}
    v(G^\ast, wa) = \begin{pmatrix}
        1 & 1 \\ 
        \lambda +1 & \lambda + 1
    \end{pmatrix} v(G,e).
\end{equation*}
The matrix is singular, so it is not possible to use this operation and still get a free semigroup.
\end{remark}
\section{Non-integer positive evaluations via clique joins}\label{sec:noninteger}

In this section we show that once the evaluation spectrum $\mathsf{S}_n(q)$ grows exponentially for every point in $(-1,0)$, it is automatically exponentially large for every non-integer positive point. In particular, this handles the problematic range $q \in (0,2)$.

The reduction uses the join operation and a simple coloring-counting identity. 

\subsection{Join with a clique shifts the evaluation point}

\begin{definition}[Join]
For graphs $G$ and $H$ on disjoint vertex sets, the \emph{join} $G\vee H$
is obtained from $G\sqcup H$ by adding all edges between $V(G)$ and $V(H)$.
\end{definition}

\begin{lemma}\label{lem:join-shift}
Let $G$ be a graph and let $m\ge 1$. Then
\begin{equation}\label{eq:join-shift}
P_{G\vee K_m}(x)=(x)_m\,P_G(x-m),
\end{equation}
where $(x)_m\coloneqq x(x-1)\cdots(x-m+1)$ is the falling factorial.
\end{lemma}

\begin{proof}
Fix an integer $x\ge m$.
In a proper $x$-coloring of $G\vee K_m$, the $m$ vertices of the clique $K_m$ must receive $m$ distinct colors,
and none of those colors may appear on $G$ (since every vertex of $G$ is adjacent to every vertex of $K_m$).
There are $(x)_m$ ways to color $K_m$, and then $P_G(x-m)$ ways to color $G$ using the remaining $x-m$ colors.
Thus \eqref{eq:join-shift} holds for all integers $x\ge m$, and hence as an identity of polynomials.
\end{proof}

\begin{corollary}\label{cor:spectral-embed}Fix a real $q$ and an integer $m\ge 1$ such that $(q)_m\neq 0$.
Then for every $n\ge 1$,
\[
|\mathsf{S}_{n+m}(q)| \ \ge\ |\mathsf{S}_n (q-m)|.
\]
In particular, for $N\ge m$ one has $|\mathsf{S}_N (q)|\ge |\mathsf{S}_{N-m}(q-m)|$.
\end{corollary}

\begin{proof}
Define $\Phi(G)\coloneqq G\vee K_m$, which maps $n$--vertex graphs to $(n+m)$-vertex graphs.
By Lemma~\ref{lem:join-shift},
\[
P_{\Phi(G)}(q)=P_{G\vee K_m}(q)=(q)_m\,P_G(q-m).
\]
Since $(q)_m\neq 0$, the map $y\mapsto (q)_m y$ is injective on real values, so distinct values of $P_G(q-m)$
yield distinct values of $P_{G\vee K_m}(q)$.
\end{proof}

\subsection{Exponential growth for every non-integer \texorpdfstring{$q>0$}{q > 0}}
We can then prove the following lower bound for the chromatic polynomial evaluation spectra at non-integer positive numbers. 
\begin{theorem}\label{thm:noninteger-positive}
Let $q>0$ be a non-integer real and set $m\coloneqq \lceil q\rceil$.
Then $q-m\in(-1,0)$ and $(q)_m\neq 0$, and for all $n\ge m+2$,
\[
|\mathsf{S}_n (q)| \ \ge\ 2^{(n-m-2)/2}.
\]

\end{theorem}

\begin{proof}
Since $q$ is not an integer, none of the factors $q, q-1, \dots, q-m+1$ vanishes, hence $(q)_m\neq 0$.
Also $m-1<q<m$ implies $q-m\in(-1,0)$.

Apply Corollary~\ref{cor:spectral-embed} with this $m$ and use Theorem~\ref{thm:negative} at the negative point $q-m$:
for $n\ge m+2$ we have $n-m\ge 2$ and hence
\[
|\mathsf{S}_n(q)| \ \ge\ |\mathsf{S}_{n-m}(q-m)|
\ \ge\ |\mathsf{S}^{\pl}_{n-m}(q-m)|
\ \ge\  2^{(n-m-2)/2}.
\qedhere \]
\end{proof}

\begin{remark}
The trick to adjoin a clique discussed in this section works well to show Theorem \ref{thm:noninteger-positive}, but it does not allow us to control certain properties of the resulting graph. In fact, if one applies the operation $B$ twice or more within a sequence of operations, then the resulting graph contains a $K_{2,3}$-subdivision, so joining by just one vertex destroys planarity.   
    \end{remark}

\section{The planar evaluation spectrum for \texorpdfstring{$q > 2$}{q > 2}}\label{sec:integers}

In this section, we prove the last bound of Theorem \ref{thm:main} for $\mathsf{S}_n(q)$ by showing the following
\begin{theorem}\label{thm:q_greater_than_2}
    Let $q > 2$ be a real number and $n \geq 3$, then
    \[\abs{\mathsf{S}^{\pl}_n(q)} \geq \sqrt{F_{n-2}} =\Theta(\varphi^{n/2}),\]
    where $(F_n)_n$ is the Fibonacci sequence and $\varphi = (1+\sqrt{5})/2$ is the golden ratio. In particular, we get the bound for integers $q \geq 3$.
\end{theorem}

\subsection{Attainable vectors}
\label{sec:attainable_int}
Let $G$ be a planar graph and $e$ one of its edges. Define the \emph{attainable vector} witnessed by $(G,e)$ as
\[
w_q{(G,e)} \coloneqq
\begin{pmatrix}
P_{G/e}(q)\\[2pt]
P_{G-e}(q)
\end{pmatrix}.
\]
We write $w(G,e) = w_q(G,e)$ when $q$ is clear from context.
We call $w$ an \emph{$n$--attainable} vector if there exists an $n$--vertex planar graph $G$ and a edge $e$ of $G$ such that $w = w(G,e)$. 

We remark that
\begin{equation}\label{eq:feasible-to-attainable}
    w_q(G,e) = \begin{pmatrix}
P_{G/e}(q)\\[2pt]
P_{G-e}(q)
\end{pmatrix}= \begin{pmatrix}
(-1)^{|V(G)|-1}Z_{G/e}(-q)\\
(-1)^{|V(G)|}Z_{G-e}(-q)
\end{pmatrix} = \begin{pmatrix}
    (-1)^{|V(G)|-1}&0\\0&(-1)^{|V(G)|}
\end{pmatrix}v_{-q}(G,e),
 \end{equation}
where $v_{-q}(G,e)$ is the feasible vector witnessed by $(G,e)$ with parameter $\lambda = -q$. In what follows, let
\begin{equation*}
        A_n =\begin{pmatrix}
            (-1)^{n-1}&0\\0&(-1)^n
        \end{pmatrix}.
    \end{equation*}
With this remark in hand, we can show a couple of lemmas analogous to those found in Section~\ref{sec:negative}. 

\begin{lemma}[Square-root lemma]\label{lem:sqrt-pos}
If there are at least $N$ distinct $n$--attainable vectors, then
\[
|\mathsf{S}^{\pl}_n(q)| \;\ge\; \sqrt{N}.
\]
\end{lemma}
\begin{proof}
    Since $A_n$ is invertible, there are at least $N$ distinct $n$-feasible vectors. By Lemma~\ref{lem:sqrt-neg}, $|\mathsf{S}^{\pl}_n(q)| \ge \sqrt{N}.$
\end{proof}
The same basic operations $B$ and $S$ defined in Section~\ref{sec:negative} will be the building blocks of our construction of exponentially many $n$--attainable vectors. 
\begin{lemma}[Matrix action]\label{lem:matrices-pos}
Let $e$ be an edge of a planar graph $G$, and write $w(G,e)=\binom{t}{u}$.
Then
\[
(w\circ S)(G,e)
=
\begin{pmatrix}
-1 & 1\\
0 & q-1
\end{pmatrix}
\binom{t}{u},
\qquad
(w\circ B)(G,e)
=
\begin{pmatrix}
q-1 & 0\\
1 & q-2
\end{pmatrix}
\binom{t}{u}.
\]
\end{lemma}
\begin{proof} Let $n = |V(G)|$ and observe that Lemma~\ref{lem:matrices-neg} and \eqref{eq:feasible-to-attainable} imply
    \[(w\circ S)(G,e)
=A_{n+1}\begin{pmatrix}
1 & 1\\
0 & -q+1
\end{pmatrix}A_n^{-1}
\binom{t}{u} = \begin{pmatrix}
-1 & 1\\
0 & q-1
\end{pmatrix}
\binom{t}{u}\]
and 
\[(w\circ B)(G,e)
=A_{n+1}\begin{pmatrix}
-q+1 & 0\\
1 & -q+2
\end{pmatrix}A_n^{-1}
\binom{t}{u} = \begin{pmatrix}
q-1 & 0\\
1 & q-2
\end{pmatrix}
\binom{t}{u}. \] 

\end{proof}

As before, we will interchangeably write, say, $S$ for the operation and the corresponding matrix. 

\subsection{Ping--pong on ratios and exponential growth}
In order to prove the case where $q\ge 3$ is an integer in Theorem~\ref{thm:main}, we will follow a similar strategy to the negative evaluation case. We choose an initial graph $G_0=K_3$, and show that operations $B$ and $D=S^2$, the concatenation of two consecutive $S$ operations, act on attainable vector ratios with disjoint images. A simple calculation then shows that there exist exponentially many $n$--attainable vectors.

Recall that for $v=\binom{x}{y}$ with $y \neq 0$, we have defined $r(v)\coloneqq x/y$. The following definition is a direct analogy of Definition~\ref{def:feasible_generated}.

\begin{definition}
 Let $a = (a_1,\dots,a_t)$ be a word in the semigroup $\{B,D\}^\ast $, $G$ be a planar graph, and $e$ be an edge of $G$. Then define 
      \[\mathbf{w}(a,G,e) \coloneqq (w\circ a_t\circ a_{t-1}\circ\dots\circ a_1)(G,e),\]
     the attainable vector witnessed by the pair generated through the sequence of operations encoded by $a$ starting with the pair $(G,e)$.
\end{definition}
We can now prove a lemma analogous to Lemma~\ref{lem:pingpong-neg}.
\begin{lemma}\label{lem:pingpong-pos}
Let $q>2$ and let $w_0\coloneqq w(K_3,e)$ where $e$ is one of the edges of $K_3$.
Then $w_0=\binom{q(q-1)}{q(q-1)^2}$ and $r(w_0)=1/(q-1)$.
Moreover, on ratios, the operations $B$ and $D =S^2$ act by
\[
f_B(r)=\frac{(q-1)r}{r+q-2},\qquad f_D(r)=\frac{r+q-2}{(q-1)^2}.
\]
Let $I\coloneqq (1/q,1)$. Then $f_B(I)\subseteq \bigl(1/(q-1),\,1)$ and $f_D(I)\subseteq \bigl(1/q,\,1/(q-1)\bigr)$,
so the images are disjoint and contained in $I$.
Consequently, if $w\neq w'$ are words in the semigroup $\{B,D\}^\ast$, then $r(\mathbf{w}(w,K_3,e))\neq r(\mathbf{w}(w',K_3,e))$.
\end{lemma}

\begin{proof}
The base computation is immediate: $K_3/e$ is a double edge, so $P_{K_3/e}(q)=q(q-1)$; $K_3-e$ is the path on three vertices, so $P_{K_3-e}(q)=q(q-1)^2$. Notice that $r(w_0)\in I$.

From Lemma~\ref{lem:matrices-pos}, if $w=\binom{x}{y}$ then
\[
Bw=\binom{(q-1)x}{x+ (q-2)y},\qquad Dw=\binom{x +(q-2)y}{(q-1)^2y},
\]
so the ratio maps are $f_B,f_D$.

Now fix $r\in I$, so $1/q<r < 1$.
Then, since $q>2$, $f_B$ is strictly increasing, so
\[f_B(r) > f_B(1/q) = \frac{1}{q-1}\quad\text{and}\quad
f_B(r)< f_B(1) = 1.
\]
Thus $f_B(I)\subseteq (1/(q-1),\,1)$.

Similarly, since $f_D$ is strictly increasing,
\[f_D(r) > f_D(1/q) = \frac{1}{q}\quad\text{and}\quad
f_D(r) < f_D(1)=\frac{1}{q-1},
\]
so $f_D(I)\subseteq (1/q,\,1/(q-1))$. 

Hence, the images are disjoint.

Finally, suppose $r(\mathbf{w}(w,K_3,e))= r(\mathbf{w}(w',K_3,e))$. 
Then their ratios agree and lie in $I$. If at least one of $w,w'$ is empty, say $w'$, then $r(\mathbf{w}(w,K_3,e)) = r(w(K_3,e))=1/(q-1) \not \in f_B(I) \cup f_D(I)$. Hence, $w$ must also be empty. 
If both $w,w'$ are non-empty, by the image disjointness just proved, the last letter of $w$ must equal the last letter of $w'$.
Since both matrices in Lemma~\ref{lem:matrices-pos} are invertible, we can cancel the last letter and iterate to conclude $w=w'$.
\end{proof}

We are now ready to prove Theorem \ref{thm:q_greater_than_2}.

\begin{proof}[Proof of Theorem~\ref{thm:q_greater_than_2}]
    Let $n \geq 3$ and $q > 2$. Consider the set $W$ of words $a \in \set{B,D}^\ast$ such that 
    \begin{equation*}
         b+2d= n-3,
    \end{equation*}
    where $b$ is the number of letters $B$ in $a$, and $d$ is the number of letters $D$ in $a$. Notice that a simple recursive argument shows that $|W| = F_{n-2}$ where $(F_n)_{n}$ is the Fibonacci sequence.
    
    Moreover, notice that every word $w\in W$ yields an $n$--vertex witness and an $n$--attainable vector $\mathbf{w}(w, K_3, e)$.
By Lemma~\ref{lem:pingpong-pos}, these $F_{n-2}$ vectors are all distinct. Applying Lemma~\ref{lem:sqrt-pos} finishes the proof:
\[
|\mathsf{S}^{\pl}_n(q)|\ge \sqrt{F_{n-2}} = \Theta(\varphi^{n/2}).\qedhere
\]
\end{proof}

\section{A planar ping--pong for \texorpdfstring{$q\in(0,2)\setminus\{1\}$}{q in (0,2) - 1}}\label{sec:planar_0_2}

In this section we show that, for every fixed $q\in(0,2)\setminus\{1\}$, one can obtain exponentially many distinct values of $P_G(q)$ already by ranging over simple, planar graphs $G$, thus completing the proof of Theorem~\ref{thm:planar}.
Unlike the clique--join reduction (Section~\ref{sec:noninteger}), the argument here preserves planarity: it only uses the local planar operations $S$ and $B$ from Section~\ref{sec:negative}. We show the following result.

\begin{theorem}\label{thm:planar_0_2}
Let $q\in(0,2)\setminus\{1\}$. Then $|\mathsf{S}^{\mathrm{pl}}_n(q)|$ grows exponentially in $n$, i.e.,
\[
|\mathsf{S}^{\mathrm{pl}}_n(q)| \;=\; 2^{\Omega_{q}(n)}.
\]
\end{theorem}

As discussed in Section~\ref{sec:intro}, we postpone the quantitative bounds until Remark~\ref{rem:q_no_dependence}. Similar to Theorem~\ref{thm:q_greater_than_2}, the proof of Theorem~\ref{thm:planar_0_2} follows a ping--pong/free-semigroup argument on the ratio of attainable vectors (see Section~\ref{sec:attainable_int} for the definition), but using \emph{higher powers} of $S$ and concatenations of $S$ and $B$.

We begin with the following lemma, which gives the ratio action of a block operation of higher powers of $S$.
\begin{lemma}
    \label{lem:S2_powers}
    For $m \geq 1$ and $q \neq 1$, we have 
    \begin{equation}\label{eq:ratio_S_2m}
        (r \circ w \circ S^{2m})(G,e) = \frac{1}{q} + \frac{(r \circ w)(G,e) - \frac{1}{q}}{(q-1)^{2m}}.
    \end{equation}
\end{lemma}
\begin{proof}
    From the computation in Lemma \ref{lem:pingpong-pos} we see that
    \begin{equation*}
        (r\circ S^2)(u)  = 
        \frac{r(u)+ q-2}{(q-1)^2}
        \end{equation*}
for every attainable vector $u$. In particular, when $u = (w \circ S^{2i})(G,e)$, $i \in \{0,1,\dots,m-1\}$, one obtains
\begin{equation*}
        (q-1)^{2i+2}(r\circ w \circ S^{2i+2})(G,e)  = 
        (q-1)^{2i}(r\circ w \circ S^{2i})(G,e)+ (q-1)^{2i}(q-2).
\end{equation*}
These equations telescope into
\begin{equation*}
            (q-1)^{2m}(r\circ w \circ S^{2m})(G,e)  = 
        (r\circ w)(G,e)+ (q-2)\frac{(q-1)^{2m}-1}{(q-1)^2-1},
\end{equation*}
which implies \eqref{eq:ratio_S_2m}.
\end{proof}

Using these block operations of higher powers of $S$, we can prove the following ping--pong lemma. 
\begin{lemma}\label{lem:ping_pong_02} Let $q \in (0,2)\setminus \{1\}$. Then, there exists a planar graph $G_0$, an edge $e_0$ of $G_0$ and block operations $K=K(q)$ and $L=L(q)$ that are each a sequence of operations in $\{B,S\}$ such that the following is true.  If, for every word $a = (a_1,\dots,a_t)$ in the semigroup $\{K,L\}^\ast $, we let 
     \[\widetilde{\mathbf{w}}(a) \coloneqq (w\circ a_t\circ a_{t-1}\circ\dots\circ a_1)(G_0,e_0),\]
     the attainable vector witnessed by the pair generated through the sequence of operations encoded by $a$ starting with the pair $(G_0,e_0)$, then for all words $w\ne w'$ in the semigroup $\{K,L\}^\ast$ with the same number of letters,  $\widetilde{\mathbf{w}}(w)\neq \widetilde{\mathbf{w}}(w')$.
    
\end{lemma}

\begin{proof}

We start by remarking that in order to prove that a tuple $(G_0,e_0,K,L)$ satisfies the desired property, it is sufficient to find disjoint sets $I$ and $J$ such that \begin{enumerate}
    \item \label{item:1}$r(w(G_0,e_0)) \in I \cup J$;    \item\label{item:2}$r((w\circ K)(G,e))\in I$, for every edge $e$ of a planar graph $G$ such that $r(w(G,e)) \in I \cup J$;
    \item \label{item:3}$r((w\circ L)(G,e))\in J$, for every edge $e$ of a planar graph $G$ such that $r(w(G,e)) \in I \cup J$.
\end{enumerate}
Indeed, let $w, w' \in \set{K,L}^t$ for some $t \geq 0$ such that $\widetilde{\mathbf{w}}(w)= \widetilde{\mathbf{w}}(w')$. In particular, $(r \circ \widetilde{\mathbf{w}}) (w) = (r \circ \widetilde{\mathbf{w}})(w')$. If $ t= 0$, then we directly get $w = w'$. Otherwise, from \eqref{item:1}--\eqref{item:3} it follows that $(r\circ \widetilde{\mathbf{w}})(w'') \in I \cup J$ for any prefix $w''$ of $w$ or $w'$. Therefore, as $I$ and $J$ are disjoint, $w$ and $w'$ must agree in their last letter by \eqref{item:2} and \eqref{item:3}. As $q \not \in \set{0,1,2}$, $S$ and $B$ correspond to invertible matrices, hence the same is true for $K$ and $L$. Hence, we may invert to cancel the action of the last letter and iterate to conclude that $w = w'$.

We divide the proof into three cases, depending on $q$.

\emph{Case 1:} $q\in (0,3/2)$, $q\ne 1$. Let \[X =(G_0,e_0,K,L,I,J)=
\Big(K_4,e,S^{2m},B^2, \big(- \infty, (q-1)^2 / (2q-3)\big], \big((q-1)^2 / (2q-3),0\big)\Big),\]
where $e_0$ is any edge of $G_0=K_4$ and $m=m(q)$ is an integer to be chosen later.
We now prove that $X$ satisfies properties \eqref{item:1}--\eqref{item:3}.
\begin{enumerate}
    \item $r(w(G_0,e_0)) = 1/(q-2)\in (-\infty,0)=I \cup J$.
    \item By Lemma \ref{lem:S2_powers}, if $r = r(w(G,e)) \in (-\infty,0) = I \cup J$, then $w_k = (w \circ S^{2m})(G,e)$ satisfies
    \begin{equation*}
        r(w_k) = \frac{r}{(q-1)^{2m}} + \frac{1}{q} \of{1 - \frac{1}{(q-1)^{2m}}} < \frac{1}{q} \of{1 - \frac{1}{(q-1)^{2m}}} \overset{m \to \infty}{\longrightarrow} - \infty.
    \end{equation*}
       Hence, for sufficiently large $m$, we get $r(w_k) \in I$.  
    \item From Lemma~\ref{lem:matrices-pos}, we obtain
    \begin{align}\label{eq:B2_matrix}
        (w \circ B^2)(G,e) &= \begin{pmatrix}
            (q-1)^2 & 0 \\
            2q-3 & (q-2)^2 
        \end{pmatrix}
        w(G,e).
    \end{align}
    If $r = r(w(G,e)) \in (-\infty,0) = I \cup J$, then $w_\ell = (w \circ B^2)(G,e)$ satisfies 
    \begin{equation*}
        \frac{(q-1)^2}{2q-3}< r(w_\ell) = \frac{(q-1)^2}{(2q-3) + (q-2)^2 r^{-1}} < 0,
    \end{equation*}
    since $2q - 3 < 0$, $r<0$, and $q \neq 1$. Thus $r(w_\ell) \in J$.
\end{enumerate}

\emph{Case 2:} $q \in (3/2, 2)$. Let \[X =(G_0,e_0,K,L,I,J)=
\Big(K_3,e,S^{2m},B^2,  \big[(q-1)^2 / (2q-3),\infty\big), (1,(q-1)^2 / (2q-3))\Big),\]
where $e_0$ is any edge of $G_0=K_3$ and $m=m(q)$ is an integer to be chosen later.
We now prove that $X$ satisfies properties \eqref{item:1}--\eqref{item:3}.

\begin{enumerate}
    \item $r(w(G_0,e_0)) = 1/(q-1)\in (1,\infty)=I \cup J$.
    \item By Lemma~\ref{lem:S2_powers}, if $r = r(w(G,e)) \in (1,\infty) = I \cup J$, then $w_k = (w \circ S^{2m})(G,e)$ satisfies 
    \begin{equation*}
        r(w_k)= \frac{1}{q} + \frac{r - \frac{1}{q}}{(q-1)^{2m}} \geq \frac{1}{q} + \frac{1 - \frac{1}{q}}{(q-1)^{2m}} \overset{m \to \infty}{\longrightarrow} \infty
    \end{equation*}
    since $1> 1/q$. Hence, for sufficiently large $m$, we get $r(w_k) \in I$.
    \item By \eqref{eq:B2_matrix}, if $r = r(w(G,e)) \in (-\infty,0) = I \cup J$, then $w_\ell = (w \circ B^2)(G,e)$ satisfies
    \begin{equation*}
        1 = \frac{(q-1)^2}{(2q-3) + (q-2)^2}\leq r(w_\ell) = \frac{(q-1)^2}{(2q-3) + (q-2)^2 r^{-1}} < \frac{(q-1)^2}{2q-3}
    \end{equation*}
    since $q>3/2$ and $r>1$. Thus 
    $r(w_\ell)\in J$.
\end{enumerate}
\emph{Case 3:} $q = 3/2$. Let \[X =(G_0,e_0,K,L,I,J)=
\big(K_4,e,S^2,B\circ S^2\circ B, (-\infty,0], (0,1/2]\big),\]
where $e_0$ is any edge of $G_0=K_4$.
We now prove that $X$ satisfies properties \eqref{item:1}--\eqref{item:3}.
\begin{enumerate}
    \item $r(w(G_0,e_0)) = 1/(q-2) = -2\in (-\infty,1/2]=I \cup J$.
    \item By Lemma~\ref{lem:S2_powers}, if $r = r(w(G,e)) \in (-\infty,1/2] = I \cup J$, then \[w_k = (w \circ S^{2})(G,e) =4r-2 \in (-\infty,0] = I.\]
    
    \item From Lemma~\ref{lem:matrices-pos}, we obtain
     \begin{equation*}
        (w \circ L) (G,e) = \begin{pmatrix}
            0 & 1/8 \\
            -1/8 & 5/16
        \end{pmatrix} w(G,e). \end{equation*}
    Thus, if $r = r(w(G,e)) \in (-\infty,1/2]= I \cup J$, then $w_\ell = (w \circ L)(G,e)$ satisfies     \begin{equation*}
        0 < r(w_u) = \frac{2}{5 - 2 r } \leq \frac{1}{2},
    \end{equation*}
    since $r\mapsto 2/(5-2r)$ is strictly increasing in $I\cup J$. Thus 
    $r(w_\ell)\in J$. \qedhere
\end{enumerate}
\end{proof}

We are now ready to prove Theorem~\ref{thm:planar_0_2}.

\begin{proof}[Proof of Theorem~\ref{thm:planar_0_2}]
Let $(G_0,e_0,K,L)$ be given by Lemma~\ref{lem:ping_pong_02} and let $k=k(q)$ and $\ell=\ell(q)$ be the number of basic operations in blocks $K$ and $L$, respectively. Let $t = \smash{ \lfloor \frac{n-|V(G_0|)}{(k+\ell)}\rfloor}$ and notice that $n_0 = |V(G_0)|+t(k+\ell) \le n$. By Lemma~\ref{lem:ping_pong_02}, there are at least
\begin{equation}\label{eq:attainable_02}
\binom{2t}{t} \ge \frac{2^{2t}}{2t+1} = 2^{\Omega_{q}(n)}  
\end{equation}
distinct $n_0$--attainable vectors, since words in $\{K,L\}^\ast$ consisting of $t$ $K$'s and $t$ $L$'s produce distinct $n_0$--attainable vectors. Thus,
\[
|\mathsf{S}^{\mathrm{pl}}_n(q)| \;\ge\; |\mathsf{S}^{\mathrm{pl}}_{n_0}(q)| \;=\;2^{\Omega_q(n)},
\]
where the left inequality comes from the fact that adding a leaf to a connected planar graph preserves planarity and alters the chromatic polynomial by a factor of $x-1$, and the right identity is given by~\eqref{eq:attainable_02} and Lemma~\ref{lem:sqrt-pos}.
\end{proof}
\begin{remark}[The dependence on $q$ or lack thereof]
    \label{rem:q_no_dependence}
    Theorem~\ref{thm:planar_0_2} as proven only gives that for every $q \in (0, 2) \setminus \set{1}$ the existence of $c(q) > 0$ and a positive integer $n_0(q)$ such that $\abs{\mathsf{S}_n^{\pl} (q)} \geq \exp\of{(c(q) +o(1)) n}$. However, a more detailed analysis easily can show that we may choose $c(q)$ to be an absolute constant, namely $1/8$. 
    \begin{itemize}
        \item For $q \in (0,1.458)$, the construction in Case~1 of Lemma~\ref{lem:ping_pong_02} works with $m = 1$. This construction provides the bound  $\abs{\mathsf{S}_n^{\pl} (q)} \geq 2^{n/4-o(n)}$;
        \item For $q \in (1,3/2]$, one can use the construction in Case~3 of Lemma~\ref{lem:ping_pong_02}. This construction provides the bound $\abs{\mathsf{S}_n^{\pl} (q) }\geq 2^{n / 6 - o(n)}$;
        
        \item For $q \in (3/2, 1.53)$, one can slightly modify the construction in Case 3 of Lemma~\ref{lem:ping_pong_02}, picking  \[(G_0,e_0,K,L,I,J)=
\big(K_4,e,S^4,B\circ S^2\circ B, (-\infty,-0.5], (-0.5, 0.56]\big).\]

This choice provides the bound $\abs{\mathsf{S}_n^{\pl} (q) }\geq 2^{n/ 8-o(n)}$;
    \item For $q\in [1.53,2]$, the construction in Case~2 of Lemma~\ref{lem:ping_pong_02} works with $m = 2$. This construction provides the bound $\abs{\mathsf{S}_n^{\pl} (q) }\geq 2^{n / 6 - o(n)}$.
    \end{itemize}
    We leave the details to the interested reader. Notice that Theorems~\ref{thm:negative},~\ref{thm:q_greater_than_2} and~\ref{thm:planar_0_2}, under the observations of Remark~\ref{rem:q_no_dependence}, imply Theorem~\ref{thm:planar}.
\end{remark}
\section{Concluding remarks} \label{sec:conclusion}

\noindent {\bf{Superexponential spectra.}} For every fixed real $q\notin\{0,1,2\}$ we proved that the evaluation spectrum
\[
\mathsf{S}_n(q)\;=\;\{P_G(q):|V(G)|=n\}
\]
grows exponentially in $n$, with explicit lower bounds of different qualities in various regimes. For positive integers $q \not\in \left\{0,1,2\right\}$, we also noted that exponential growth for the evaluation spectrum of $P_{G}(q)$ is best possible (up to the correct base of the exponent), due to the fact that $|\mathsf{S}_n(q)| \leq q^{n}+1$. Nevertheless, at other values of $q$, the size of $\mathsf{S}_n(q)$ could be significantly larger, in principle. For example, when $q=-1$, $|P_{G}(-1)|$ counts the number of acyclic orientations of $G$, so the natural upper bound for $|\mathsf{S}_n(-1)|$ is $|P_{K_n}(-1)|=n!$, rather than $q^n+1$. In this direction, we would like to make the appealing following conjecture for the spectrum of acyclic orientations. 

\begin{conjecture}\label{conj:superexp_negative}
$$|\mathsf{S}_n(-1)|\ \text{is superexponential in\ } n.$$ 
\end{conjecture}

A closely related ``spectrum'' question concerns the unrestrained spanning--tree spectrum $\mathcal{T}(n)$. We discussed the recent works of Chan--Kontorovich--Pak and of Alon--Buci\'c--Gishboliner~\cite{ChanKontorovichPak24trees,ABG} which show that $|\mathcal{T}^{\pl}(n)|$ is exponential in $n$, if one restricts the count to planar graphs. In the unrestricted setting, however, the extremal values are also superexponential (recall Cayley's theorem that $\tau(K_n)=n^{n-2}$), so it is a natural to conjecture that $|\mathcal{T}(n)|$ is also superexponential. In this direction, we would like to reiterate the very appealing \cite[Conjecture 5.2]{ChanKontorovichPak24trees}.

\begin{conjecture}\label{conj:superexp_tau}
$$|\mathcal{T}(n)| = e^{\Omega(n \log n)}.$$
\end{conjecture}

\noindent {\bf{How far does the ping--pong method go?}} 
Our arguments use remarkably little about chromatic polynomials beyond the existence of a two--term deletion--contraction relation
and the ability to realize local operations as linear maps on a small ``boundary'' space.
This suggests that similar results should be possible for a broader class of Tutte--type invariants.
For example, the Tutte polynomial $T_G(x,y)$ (and equivalently the Potts/random--cluster partition function) is universal among
Tutte--Grothendieck invariants and is governed by deletion--contraction~\cite{Tutte1954}.
It would be interesting to study evaluation spectra such as
\[
\{\,T_G(x_0,y_0):|V(G)|=n\,\}
\qquad\text{or}\qquad
\{\,Z_G(q_0,v_0):|V(G)|=n\,\}
\]
at fixed parameters, and to understand when the same feasible--vector/ping--pong mechanism yields exponential (or even superexponential) growth. Here $Z_G(q,v)$ denotes the random–cluster (Potts) partition function, a standard specialization of the Tutte polynomial. In particular, it would be interesting to determine the specializations of the Tutte polynomial (or more general Tutte--Grothendieck invariants) that guarantee
exponential growth of the corresponding evaluation spectrum, as well as the exceptional locus where the spectrum is subexponential.

\end{document}